\documentclass[11pt]{article}
\usepackage[utf8]{inputenc}
\usepackage{lmodern}
\usepackage{subfiles}
\usepackage{enumitem}
\setenumerate{topsep=6pt,ref={\normalfont(\roman*)},label={\normalfont(\roman*)}, itemsep=0pt} 

\usepackage{amsfonts}
\usepackage{amsthm}
\usepackage{amsmath}
\usepackage{amssymb}
\usepackage{amscd}
\usepackage{mathrsfs}
\usepackage{mathtools}
\usepackage{bbm}
\usepackage{esint}

\usepackage[margin=2.7cm]{geometry}
\usepackage{setspace}
\usepackage{indentfirst}
\usepackage{graphicx}
\usepackage{graphics}
\usepackage{lscape}
\usepackage{pgf,tikz}
\usepackage{tikz-cd}
\usepackage{color}
\usepackage{pict2e}
\usepackage{epic}
\usepackage{epstopdf}
\usepackage{titlesec, titlefoot}
\titleformat{\section}[block]{\Large\bfseries\filcenter}{\thesection}{1em}{}
\usepackage{commath}
\usepackage{float}
\usepackage{caption}
\usepackage{etoolbox}
\usepackage[affil-it]{authblk}
\usepackage{combelow}

\usepackage[hidelinks, bookmarksdepth=3,citecolor=blue,colorlinks]{hyperref}
\hypersetup{bookmarksopen=true} 
\usepackage{hypcap}

\graphicspath{{./Pictures/}}
\allowdisplaybreaks

\expandafter\def\expandafter\normalsize\expandafter{%
\normalsize
\setlength\abovedisplayskip{6pt}
\setlength\belowdisplayskip{6pt}
\setlength\abovedisplayshortskip{6pt}
\setlength\belowdisplayshortskip{6pt}
}

\theoremstyle{plain}

\renewcommand*\thesection{\arabic{section}}
\numberwithin{equation}{section} 

\newtheorem{theorem}{Theorem}[section]
\newtheorem{lemma}[theorem]{Lemma}
\newtheorem*{lemma*}{Lemma}
\newtheorem{proposition}[theorem]{Proposition}
\newtheorem{corollary}[theorem]{Corollary}

\theoremstyle{definition}
\newtheorem{definition}[theorem]{Definition}
\newtheorem{remark}[theorem]{Remark}

\newtheorem{question}[theorem]{Question}

\expandafter\let\expandafter\oldproof\csname\string\proof\endcsname
\let\oldendproof\endproof
\renewenvironment{proof}[1][\proofname]{%
\oldproof[\upshape \bfseries #1]%
}{\oldendproof}

\makeatletter
\def\@makechapterhead#1{%
\vspace*{50\p@}%
{\parindent \z@ \raggedright \normalfont
\interlinepenalty\@M
\Huge\bfseries  \thechapter.\quad #1\par\nobreak
\vskip 40\p@
}}
\makeatother

\newcommand{\eps}{\varepsilon}

\DeclareMathOperator{\Sym}{Sym}

\DeclareMathOperator{\ind}{ind}

\def \a{\alpha}

\def \R {\mathbb{R}}
\def \Z {\mathbb{Z}}

\def \N{\mathbb{N}}
\def \D{\textup{D}}

\def \e{\varepsilon}
\def \d{\,\textup{d}}

\def \p{\partial}

\def \mb{\mathbb}

\def \tp{\textup}

\def \id{\textup{id}}
\def \loc{\textup{loc}}

\begin{document}

	\title{\textbf{Constancy of the index for gradient mappings}}
	
	
	\author[1]{{\Large Andr\'e Guerra}}
	\author[2]{{\Large Riccardo Tione}}
		
	\affil[1]{\small Department of Pure Mathematics and Mathematical Statistics,  University of Cambridge,\protect\\  Wilberforce Rd, Cambridge CB3 0WB, UK
	\protect\\
	{\tt{adblg2@cam.ac.uk}} \vspace{1em} \ }
		
	\affil[2]{\small Dipartimento di Matematica, Universit\`a degli Studi di Torino,
Via Carlo Alberto 10, 10123 Torino, Italy 
	\protect\\
	{\tt{riccardo.tione@unito.it}}  }
	
	\date{}
	
	\maketitle

	\begin{center}
	\vspace{-0.4cm}
	\textit{ To Jan Kristensen, with friendship and admiration. \vspace{1em} }
	\end{center}	
	
	\unmarkedfntext{
	\hspace{-0.75cm}
	\emph{Acknowledgments.} AG acknowledges the support of Dr. Max R\"ossler, the Walter Haefner Foundation and the ETH Z\"urich Foundation, as well as the support of the Royal Society through a Newton International Fellowship. He also acknowledges the hospitality of Instituto Superior T\'ecnico and the Universit\`a di Torino, where part of this research was conducted.
	}
	
	
	\begin{abstract}
We show that if the Hessian of a $C^{1,1}$ function has uniformly positive determinant almost everywhere then its index is locally constant, as conjectured by \v Sver\'ak in 1992.\ We deduce this result as a consequence of a more general theorem valid for quasiregular gradient mappings.
	\end{abstract}

	\section{Introduction}	

Geometric Function Theory studies how the analytic properties of maps influence their geometric and topological behavior. 
A prototypical example of this interplay is the Inverse Function Theorem: if a $C^1$ map $f\colon \R^n\to \R^n$ fulfills $\det \D f(x_0) \neq 0$, then $f$ is a local homeomorphism at $x_0$. In applications one is often interested in maps with lower regularity, for instance Lipschitz maps. Although the Inverse Function Theorem fails for Lipschitz maps, it is restored if we additionally assume that the map satisfies
\begin{equation}
\label{eq:non-deg}
\det \D f(x)>0 \quad \tp{ for a.e.\ $x$} \tp{ in } \R^n;
\end{equation}
recall that, by Rademacher's Theorem, Lipschitz maps are a.e.\ differentiable.
In this case $f$ has an a.e.\ inverse which moreover is a Sobolev map \cite{Fonseca1995}. 



Now suppose $f$ is a \textit{gradient map}, i.e.\ $f=\D u$ for some scalar $u\colon \R^n\to \R$. If $f\in C^1(\R^n,\R^n)$ and $\det \D f(x_0) \neq 0 $, then the symmetric matrix $\D f=\D^2 u$ has a constant number of negative eigenvalues around $x_0$: this is the \textit{index} of $\D^2 u$, which we denote by $\ind(\D^2 u)$. Indeed, for any $k\in \N$, the set of symmetric matrices with index $k$ is a connected component of the space of non-singular symmetric matrices; around $x_0$, by continuity $\D f$ takes values in only one such component and so its index is constant. This simple argument fails when $f$ is merely Lipschitz, as $\D f$ becomes discontinuous. Here we show that, nevertheless, the same conclusion holds:

	\begin{theorem}\label{thm:MA}
	Let $\Omega\subset \R^n$ be open and connected. If $u\in W^{2,\infty}_\loc(\Omega)$ is such that
	\begin{equation}
	\label{eq:MA}
	\det \D^2 u(x) \geq \delta>0 \quad \text{for a.e.\ } x\in \Omega,
	\end{equation}
	then $\ind(\D^2 u)$ is constant a.e.\ in $\Omega$. 	The same conclusion holds if, instead of \eqref{eq:MA}, we have $\det \D^2 u \leq - \delta<0$ a.e.\ in $\Omega$.
	\end{theorem}

	Theorem \ref{thm:MA} proves in particular a conjecture made by \v{S}ver\'ak in \cite[p.\ 297]{Sverak1992}, who had previously proved this result when $n\leq 3$ \cite{Sverak1991b,Sverak1992}. We will momentarily discuss  this dimensional restriction, as well as related literature, but first let us note that assumption \eqref{eq:MA} is optimal:

	
	\begin{remark}[Optimality]
Even when $n=2$ assumption \eqref{eq:MA} cannot be relaxed to \eqref{eq:non-deg} with $f=\D u$, as observed in  \cite{Lewicka2017a}.
Indeed,  $u(x_1,x_2)=x_1^3/|x_1| e^{x_2^2/2}$ satisfies \eqref{eq:non-deg} over $\mb B^2$, is in $W^{2,\infty}(\mb B^2)$, but $\ind( \D^2 u) = - 2 \chi_{\{x_1<0\}}$ a.e.\ in $\mb B^2$. 

	\end{remark}
	

	As stated, Theorem \ref{thm:MA} is purely analytic in nature, but in fact it is deeply connected to a topological property of gradient maps. To explain this connection we need to recall some basic ideas from Morse theory \cite{Milnor1963}. A central result in this theory is that the index of a non-degenerate critical point $x_0$ of a smooth function $u$ can be inferred from the change in topology of its sublevel sets as one crosses the corresponding critical value. For instance, suppose $c$ is a critical value of $u$ associated with a single critical point $x_0$ of index $k$, and assume that $u$ is defined on the $n$-torus (rather than on $\R^n$). Then, for all sufficiently small $\e>0$, the sublevel set $\{ u\leq c+\e\}$ is diffeomorphic to $\{u\leq c-\e\}$ with a $k$-handle attached  (see e.g.\ \cite[p.\ 17]{Milnor1963}):
	\begin{equation}
\label{eq:diffeo}
\{ u \leq c + \e \} \cong \{u\leq c- \e \}\cup \tp{($k$-handle)}.
\end{equation}
This assertion in general fails when $u\not \in C^2$. It is therefore natural to seek a robust topological invariant that can capture the change in topology of sublevel sets even when $u$ is not twice differentiable.
In the course of our proof of Theorem \ref{thm:MA} we show 
that certain homology groups, defined via sublevel sets of $u$, are constant, and this topological perspective is a key ingredient in our approach.

	Although Theorem \ref{thm:MA} is trivial for smooth functions, we will deduce it as a consequence of a result which is non-trivial even in the smooth category. For $K>0$, consider the differential inclusion
	\begin{equation}
\label{eq:DI}
\D f(x) \in Q_K \quad \text{for a.e.\ } x\in \Omega, \qquad Q_K \equiv \{A\in \Sym(n): |A|^n \leq K \det A\},
\end{equation}
where $\Sym(n)$ denotes the space of $n\times n$ symmetric matrices. Notice that if $u\in C^{1,1}(\Omega)$ solves \eqref{eq:MA} then $\D^2 u \in Q_K$ a.e.\ in $\Omega$, where $K =\e^{-1} \tp{Lip}(\D u)^n.$ Hence Theorem \ref{thm:MA} is a consequence of the following:

	\begin{theorem}\label{thm:main}
	If $f\in W^{1,n}_\loc(\Omega,\R^n)$ solves \eqref{eq:DI} then $\ind(\D f)$ is constant a.e.\ in $\Omega$.
	\end{theorem}		

Solutions to \eqref{eq:DI} are sometimes referred to as \textit{quasiregular gradient maps} \cite{Kovalev2005}, since this differential inclusion asserts that locally $f$ is a gradient map which is also quasiregular in the sense of Geometric Function Theory \cite{Rickman1993}. Besides their intrinsic interest, quasiregular gradient maps play an important role in the study of elliptic PDEs in the plane, cf.\ \cite[\S 16]{Astala2009} or \cite{Baernstein2005}. We emphasize that, even if $f$ is assumed to be smooth, the conclusion of Theorem \ref{thm:main} is far from obvious, and in fact its proof relies crucially on the properties of the branch set of quasiregular maps (or more generally maps of finite distortion). In this direction we refer the reader to \cite{Kovalev2011a,Kovalev2010} for very interesting results concerning lack of branching for solutions to differential inclusions.
	
\subsection{Previous results and related literature} 	
Both Theorems \ref{thm:MA} and \ref{thm:main} have been known for several decades when $n\leq 3$. The first proof of the latter result appeared in an influential manuscript \cite{Sverak1991b} due to \v Sver\'ak, which was unfortunately never published\footnote{The interested reader can find essentially \v Sver\'ak's proof in \cite[\S 7]{Lewicka2017a}.}. In any dimension, \v Sver\'ak showed that if $u\in W^{2,\infty}_\loc(\Omega)$ is a solution of \eqref{eq:MA} then 
\begin{equation}
\label{eq:Sv}
\lvert\{\ind(\D^2 u)=0 \}\rvert>0 \quad \implies \quad \ind(\D^2 u)=0 \tp{ a.e.\ in } \Omega,
\end{equation}
and $u$ is then locally \textit{convex}.
By replacing $u$ with $-u$, the same conclusion holds also with $0$ replaced by $n$. Since, for $n\leq 3$, there are only two possible values of the index (once one restricts to matrices whose determinant has a sign), Theorem \ref{thm:MA} follows in those dimensions from \eqref{eq:Sv}.

In a subsequent paper \cite{Sverak1992}, \v Sver\'ak introduced the integrands 
$F_k(A)\equiv  \lvert \det(A) \rvert \chi_{\ind^{-1}(\{k\})}(A),$ which he proved to be quasiconvex in $\Sym(n)$. This quasiconvexity  quickly implies Theorem \ref{thm:MA} under the strong \textit{additional assumption} that $u$ is affine in a neighborhood of $\partial \Omega$. \v Sver\'ak's integrands play an important role in the vectorial Calculus of Variations, in part due to their extremal properties \cite{Guerra2018}, and have been used  to construct counterexamples to regularity  \cite{Muller2003}.  

More recently, Faraco, Kirchheim and Sz\'ekelyhidi \cite{Faraco2008,Kirchheim2008,Szekelyhidi2005} pioneered a deep and very general separation theory for differential inclusions in $\R^{2\times 2}$; see also \cite{DePhilippis2023,Faraco2012,Lamy2019,Lamy2024,Sverak1993,Zhang1997} for further results in this setting. Their theory gives a fundamentally new proof of both Theorems \ref{thm:MA} and \ref{thm:main} when $n = 2$, but already when $n=3$ it is unclear how to employ their methods to the same effect. For the case $n = 2$, see also \cite[Theorem 1.4]{Iwaniec2008} for yet another proof of Theorem \ref{thm:main} under additional conditions. Still for $n=2$, we also refer the reader to \cite{Cao2023,Conti2012,Inauen2025} for  further flexibility and rigidity results for very weak solutions of the Monge--Amp\`ere equation.
\ 

Despite the remarkable progress described above, both Theorem \ref{thm:MA} and \ref{thm:main} remained open in general dimension. Theorem \ref{thm:MA} is arguably the first complete separation result for differential inclusions in higher dimensions.

\subsection{Elements of the proof}
\label{sec:proof}

As explained above, in low dimensions it suffices to consider Hessians which have index 0 (or $n$) in a set of positive measure. At a point $x_0$ of twice differentiability we have $\ind(\D^2 u(x_0))=0$ (resp.\ $\ind(\D^2 u (x_0))=n$) if and only if $x_0$ is a minimum (resp.\ maximum) of $u$. By touching the graph of $u$ from below with hyperplanes, it is not too difficult to show that $\ind(\D^2 u)=0$ in a neighborhood of a Lebesgue point where the same condition holds, provided $\D u$ is injective near $x_0$, see \cite[Lemma 3]{Sverak1991b} or \cite[Lemma 6.12]{Lewicka2017a}. To infer from this the full result, in  \cite{Sverak1991b} \v Sver\'ak  relies  in a crucial way on a beautiful characterization of strictly convex functions due to Ball \cite{Ball1980a}.

The main novelty in this paper is that we deal with \textit{saddle points}: these are much more challenging to handle than minima or maxima, since it is not even clear how to define the index of a saddle point $x_0$ if $\D^2 u(x_0)$ does not exist, while minima and maxima are meaningful even for continuous functions. Nonetheless, we also offer a simpler approach to study minima and maxima, and in our way to Theorem \ref{thm:main} we give a new proof of Ball's theorem.
 
In order to explain our main idea let us assume, after subtracting an affine function, that $u(x_0)=0$ and $\D u(x_0)=0$, for a function $u\in W^{2,\infty}_\loc(\Omega)$ satisfying \eqref{eq:MA}. As explained below \eqref{eq:diffeo}, a key challenge in our problem is to find a topological invariant which, on the one hand, encodes the change of topology between the sublevel sets $\{u\leq \delta\}$ and $\{u\leq -\delta\}$, for $\delta>0$ small, and on the other hand is still well-defined (and well-behaved!) when $u$ is not smooth.  Note that in the case where $x_0$ is a minimum this ``invariant'' is clear, since the latter set is empty. It turns out that the classical notion of \textit{critical groups}, which in the smooth setting goes back at least to Bott's work \cite{Bott1954}, provides a suitable invariant even in low regularity. The critical groups $C_k(u,x_0)$ of $u$ at $x_0$ are (essentially) the relative homology groups 
$$H_j(\{u\leq  \delta \},\{u\leq  -\delta\})$$ with integer coefficients, which are well-defined even if $u$ is merely continuous.  The main observation of this paper is that the critical groups are well-behaved whenever $\D^2 u$ is quasiregular: on one hand, we show that if $x_0$ is a point of twice differentiability of $u$ then
$$C_j(u,x_0)= \delta_{j,k} \Z, \qquad k = \ind \D^2 u(x_0),$$
hence the critical groups capture the index of $\D^2 u$; on the other hand,  we prove that they are independent of $x_0$ for $x_0$ in an open and connected subset of $\Omega$ with full measure. From these assertions we conclude that the index of $\D^2 u$ is necessarily a.e.\ constant in $\Omega$.





\section{A brief review of mappings of finite distortion}


We recall that $f\in W^{1,n}_\loc(\Omega,\R^n)$ is said to be a \textit{map of finite distortion} if there is a measurable function $K\colon \Omega\to [0,+\infty]$ such that $|\D f(x)|^n \leq K(x) \det \D f(x)$ and $K(x)<\infty$ for a.e.\ $x\in \Omega$. For such maps, we define the \textit{distortion function}
$$K_f(x) \equiv \begin{cases} \frac{|\D f(x)|^n}{\det \D f(x)} & \tp{if } \det \D f(x)\neq 0,\\
1 &\tp{otherwise.}
\end{cases}.$$
We say that $f$ is quasiregular if $K_f\in L^\infty(\Omega)$.

A continuous map $f\colon \Omega\to \R^n$ is said to be \textit{locally injective} at $x$ if there is a neighborhood $\Omega'\subseteq \Omega$ of $x$ such that $f|_{\Omega'}$ is injective. We write 
$$B_f\equiv \{x\in \Omega: f \text{ is not locally injective at } x\}$$
for the \textit{branch set} of $f$, and we say that $f$ is \textit{locally injective} if $B_f=\emptyset$. We also say that $f$ is \textit{open} if it maps open sets to open sets, and \textit{discrete} if its pre-image of any point is a discrete set.

We are interested in conditions which guarantee that a mapping has a ``small'' branch set, both in a topological and in measure-theoretic sense. The following result is well-known:

\begin{theorem}\label{thm:FD}
Let $f\in W^{1,n}_\loc(\Omega,\R^n)$ be a map of finite distortion. Then $f$ is continuous and a.e.\ differentiable.
Suppose, in addition, that $f$ is non-constant and that $K_f\in L^p_\loc(\Omega)$, where
\begin{equation}
\label{eq:condp}
p>n-1 \text{ if } n>2, \qquad p\geq 1 \text{ if } n=2.
\end{equation} 
Then  $\det \D f>0$ a.e.\ in $\Omega$ and $\Omega\setminus B_f\text{ is an open, connected set with full measure in } \Omega.$
\end{theorem}

\begin{proof}
For the continuity and differentiability see  Remark 2.22 and Corollary 2.25 in \cite{Hencl2014a}. The condition on the distortion guarantees that $f$ is open and discrete \cite[Theorem 3.4]{Hencl2014a}, and for such maps $B_f$ has topological dimension at most $n-2$ \cite[Theorem 3.2]{Hencl2014a}, hence $\Omega\setminus B_f$ is still a connected set.\ Concerning the fact that $|B_f|=0$, we note that by \cite[I.4.11]{Rickman1993} we have $\det \D f(x_0)=0$ at every differentiability point $x_0\in B_f$. On the other hand,  the condition on the distortion also implies that $\det \D f>0$ a.e.\ in $\Omega$ \cite[Theorem 4.13]{Hencl2014a}, and since by the first part a.e.\ point is a point of differentiability, the conclusion follows.
\end{proof}

	\section{A brief review of relative homology}	
		
In this section we give a brief review of relative homology, referring the reader to \cite{Hatcher2002} for more details. A \textit{pair of spaces} $(X,A)$ is a topological space $X$ such that $A\subseteq X$. 
A \textit{map of pairs} $f\colon (X,A)\to (Y,B)$ is a continuous map $f\colon X\to Y$ such that $f(A)\subseteq B$. We will also need to consider \emph{homotopies} $h\colon [0,1]\times (X,A)\to (Y,B)$, namely maps such that $h(t,\cdot)$ is a map of pairs for all $t \in [0,1]$.

Given a pair of spaces $(X,A)$, one defines $C_k(X,A)\equiv C_k(X)/C_k(A)$, where $C_k(X)$ denotes the group of $k$-chains in $X$ with integer coefficients. The boundary operator $\p\colon C_k(X)\to C_{k-1}(X)$ maps $C_k(A)$ to $C_{k-1}(A)$ and it descends to the quotient $C_{k}(X,A)$, preserving $\p\circ \p=0$. Thus there is a chain complex
$$\dots  \stackrel{}{\longrightarrow} C_{k+1}(X,A) \stackrel{\p}{\longrightarrow} C_k(X,A) \stackrel{\p}{\longrightarrow} C_{k-1}(X,A)  \stackrel{}{\longrightarrow}\dots $$
whose homology groups $\ker \p/\tp {im\,} \p$ are the \textit{relative singular homology groups} $H_k(X,A)$. 
Classes in $H_k(X,A)$ are represented by $k$-chains $\a\in C_k(X)$ with $\p \a \in C_{k-1}(A)$, and they are trivial if and only if $\a = \p \beta + \gamma$ for $\beta\in C_{k+1}(X)$ and $\gamma\in C_k(A)$.
%
We also note that every map $f\colon (X,A)\to (Y,B)$ induces in a natural way a homomorphism $f_*\colon H_k(X,A)\to H_k(Y,B)$.  
%

The following properties, sometimes called the Eilenberg--Steenrod axioms, are useful:
\begin{enumerate}[label=(\alph*)]
\item\label{it:dim}\textit{Dimension axiom}: $H_k(\{p\})= \delta_{k,0}\Z$, where $\delta_{i,j}$ is the Kronecker delta;
\item\label{it:*}  \textit{Functoriality and naturality:} $(\id_X)_*=\id$,  $(g\circ f)_* = g_* \circ f_*$ and $\p\circ f_* = (f|_A)_*\circ \p$;
\item\label{it:inv} \textit{Homotopy invariance}: if $f_0,f_1\colon (X,A)\to (Y,B)$ are homotopic, i.e.\ if there is $h\colon [0,1]\times (X,A)\to (Y,B)$ such that $h(0,\cdot) = f_0$ and $h(1,\cdot) =f_1$, then $(f_0)_*=(f_1)_*$;
\item\label{it:excision} \textit{Excision}: if $W\subset X$ is open and $\overline W \subseteq \tp{int}(A)$ then the inclusion $i\colon (X\setminus W, A\setminus W) \to (X,A)$ induces an isomorphism $i_*$;
\item\label{it:exact} \textit{Exactness}: if $ B \subseteq A\subseteq X$, the inclusions $i\colon (A,B)\to (X,B)$ and $j\colon (X,B)\to (X,A)$ induce an \emph{exact} sequence:
$$\dots \stackrel{}{\longrightarrow} H_{k+1}(X,A) \stackrel{\p}{\longrightarrow} H_k(A,B)  \stackrel{i_*}{\longrightarrow} H_k(X,B)  \stackrel{j_*}{\longrightarrow} H_k(X,A) \stackrel{}{\longrightarrow} \dots  $$
Recall that a sequence is \emph{exact} if the image of any map of the sequence equals the kernel of the next map in the sequence.
\end{enumerate}	
	 	
	
	Given a subset $Y\subseteq X$, we say that $Y$ is a \textit{deformation retract} of $X$ if there is a homotopy $h\in C([0,1]\times X,X)$ such that $h(0,\cdot)= \id_X$, $h(1,X)\subseteq Y$ and $h(t,\cdot)|_{Y} = \id|_Y$ for any $t\in [0,1]$. The next lemma is then easy to deduce from the above properties:
	
	\begin{lemma}\label{lemma:ret}
	Suppose that $A'\subseteq A\subseteq X'\subseteq X$.
	\begin{enumerate}
	\item\label{it:def1} If $X'$ is a deformation retract of $X$, then $H_k(X,A) \cong H_k(X',A)$ for all $k$;
	\item\label{it:def2} If  $A'$ is a deformation retract of $A$, then $H_k(X,A) \cong H_k(X,A')$ for all $k$.
	\end{enumerate}
	\end{lemma}
	
		\begin{proof}
	We first prove \ref{it:def1}, so let $h$ be the corresponding homotopy and $i\colon X'\to X$ be the inclusion map.\ Let $r \equiv h(1,\cdot)
\colon (X,A)\to (X,A)$ and let $\tilde r\colon (X,A) \to (X',A)$ simply be defined by $\tilde r(x) = r(x)$ for all $x \in X$.\ Then we have $\tilde r \circ i = \id_{X'}$, and hence by property \ref{it:*} that $\tilde r_* \circ i_* = \id$. Since $i \circ \tilde r = r$, by properties \ref{it:*} and \ref{it:inv} we have $i_*\circ \tilde r_* = r_* = \id$. This yields the fact that $i_*: H_k(X',A) \to H_k(X,A)$ is a isomorphism of groups for every $k$.
	Then \ref{it:def2} follows from property \ref{it:exact}, since by \ref{it:def1} we have $H_k(A,A')\cong H_k(A',A')=\{0\}$.
	\end{proof}

	In this paper we will only need to know concretely the relative homology in the following very simple instance, see e.g.\ \cite[Example 2.17]{Hatcher2002} or \cite[p.\ 172]{Mawhin1989}:
	
	\begin{lemma}\label{lemma:computationhom}
	Let $\mb B^n\subset \R^n$ be the unit ball. Then $H_k(\mb B^n, \mb B^n\backslash \{0\}) \cong \delta_{k,n} \Z.$	
	\end{lemma}
%
%
%
%
				
	\section{Local Morse theory and critical groups}
\label{sec:morse}

In this section we establish some Morse-theoretic results for functions with isolated critical points. We are interested in the \textit{local setting} (the function is defined only on an open subset $\Omega\subset \R^n$) in \textit{low regularity} (the function is only $C^1$), and so we present complete proofs for the sake of readability, but we refer the reader to the monographs \cite{Chang1993,Mawhin1989,Milnor1963} for analogous results.

The fundamental idea of Morse theory is that the topology of the sublevel sets of a function can only change when one crosses a critical value, as otherwise the sublevel sets can be deformed into each other using a (pseudo)-gradient flow.  As one crosses a critical value, the change in topology can be captured by the \textit{critical groups}:

	
	\begin{definition}[Critical group]\label{def:crit}
	Let $u\in C^1(\Omega)$, let $x_0\in \Omega$ and write $c\equiv u(x_0)$. Given  $U\subseteq \Omega$ a neighborhood\footnote{That is, $U$ is a set such that there is $\delta>0$ with $B_\delta(x_0)\subset U$.} of $x_0$ and $k\in \{0,\dots,n\}$,  we set
	$$C_{k}(u,x_0)\equiv H_k(\{u\leq c\} \cap U, \{u\leq c\}\cap U\setminus \{x_0\} ).$$
	\end{definition}
	
One can readily check that the definition of critical group does not depend on $U$: if $V$ is another neighborhood of $x_0$ then, by the excision property \ref{it:excision} applied with $X=\{u\leq c\}\cap U$, $A=X\setminus \{x_0\}$ and $W=X\setminus V$, one finds an isomorphism between the critical groups computed with respect to $U$ or to $U\cap V$.	In order to simplify the notation, we now suppose $\Omega=\mb B^n$ and that $u\in C^1(\mb B^n)$ satisfies
\begin{equation}
\label{eq:assum}
u(0)=0, \qquad \{x\in \mb B^n: \D u(x)=0\}=\{0\}.
\end{equation}	
Since $u$ is only $C^1$, we need the following:

\begin{lemma}[Pseudo-gradient vector field]\label{lemma:pseudo}
Given $u\in C^1(\mb B^n)$ satisfying \eqref{eq:assum}, there is a vector field $X\in C^{\infty}(\mb B^n\setminus \{0\}, \R^n)$ such that, in $\mb B^n\setminus \{0\}$, we have
\begin{equation}
\label{eq:pgvf}
|X|\leq 2 |\D u|\quad  \text{ and }\quad |\D u|^2 \leq \langle \D u, X\rangle.
\end{equation}
\end{lemma}

\begin{proof}
Since $\D u(x)\neq 0$ if $x\in \mb B^n\setminus \{0\}$, for any such $x$ one can find a vector $w_x\in \R^n$ such that the inequalities in \eqref{eq:pgvf} hold at $x$ with $X(x)=w_x$; by continuity of $\D u$, they must in fact hold in a neighborhood $W_x$ of $x$.  Taking a locally finite cover $\{W_{x_i}\}_{i\in I}$ of $\mb B^n\setminus \{0\}$ with associated partition of unity $\{\varphi_i\}_{i\in I}$, we set $X\equiv \sum_{i\in I} \varphi_{i} w_{x_i}. $
\end{proof}
	

In particular, we can define a pseudo-gradient flow $\phi_t \equiv \phi(t,\cdot)\colon \mb B^n\backslash \{0\}\to \R^n$ as the unique solution of the Cauchy problem
\begin{equation}
\label{eq:flow}
\begin{cases}
\frac{\d}{\d t} \phi_t(x) = - \frac{X(\phi_t(x))}{|X(\phi_t(x))|^2},\\
\phi_0(x) =x,
\end{cases}
\end{equation}
which is defined in a maximal interval of existence $t\in \left]t_-(x),t_+(x)\right[$. The choice $-X/|X|^2$ is convenient as it ensures that the flow reaches every point, including the origin, in finite time. Before giving the proof of the main Lemmas \ref{lemma:ngbd}-\ref{lemma:def}, it is better to record a few facts about (pseudo)-gradient flows of functions fulfilling \eqref{eq:assum}. This will be done in the next Subsection \ref{sub:GTR}, while the proof of the main Lemmas is the content of Subsection \ref{sub:male}. These will be instrumental to get the crucial properties of the critical groups of Subsection \ref{sub:crit}.

\subsection{Some general results}\label{sub:GTR}
 
The next two lemmas provide simple but important estimates.
 
 \begin{lemma}\label{lemma:tec1}
 	Let $u\in C^1(\mb B^n)$ satisfy \eqref{eq:assum}. Then, for any $x \in \mb B^n$ and any $0 < t \le s < t_+(x)$, we have:
 	\begin{equation}
 		\label{eq:cauchyuphi}
 		u(\phi_{s}(x)) - u(\phi_{t}(x))\leq - \frac{s-t}{4}.
 	\end{equation}
 	In particular, $t_+(x) <+ \infty$ for all $x \neq 0$.
 \end{lemma}
 \begin{proof}
 To see \eqref{eq:cauchyuphi}, we write
 \[
 u(\phi_{s}(x)) - u(\phi_{t}(x)) =  \int_{t}^{s} \frac{\d}{\d z} u(\phi_z(x)) \d z\overset{\eqref{eq:flow}}{=} -\int_{t}^{s} \frac{ \langle \D u (\phi_z(x)), X( \phi_z(x))\rangle}{|X(\phi_z(x))|^2} \d z\overset{\eqref{eq:pgvf}}{\leq} - \frac{s-t}{4}.
 \]
 Due to the boundedness of $u$, we infer $t_+(x) <+ \infty$ for all $x \neq 0$.
 \end{proof}
 
	\begin{lemma}\label{lemma:tec2}
		Let $u\in C^1(\mb B^n)$ satisfy \eqref{eq:assum} and let $K$ be a compact set such that $0 \notin K$.
		If we have $\phi_t(x) \in K$ for all $t_1\le t \le t_2$, then 
		\begin{equation}
			\label{eq:cauchygen}
			|\phi_{t_2}(x)- \phi_{t_1}(x)| \leq \frac{t_2-t_1}{\min_{K}|Du|}.
		\end{equation}
	\end{lemma}
	\begin{proof}
We get \eqref{eq:cauchygen} by estimating
		\[
		|\phi_{t_2}(x)- \phi_{t_1}(x)|= \left\lvert \int_{t_2}^{t_1}  \frac{\d}{\d s} \phi_s(x) \d s\right\rvert \leq \int_{t_1}^{t_2} \frac{\d s}{|X(\phi_s(x))|} \leq \frac{t_2-t_1}{\min_{K}|Du|}
		\]
		using \eqref{eq:pgvf}--\eqref{eq:flow}.
	\end{proof}
	
	\begin{corollary}\label{cor:tec2}
			Let $0 < r < 1$ and set
		\begin{equation}\label{eq:delta}
			\delta(r)\equiv \min_{\overline{B_r(0)}\setminus B_{r/2}(0)} |\D u|\overset{\eqref{eq:assum}}{>}0.
		\end{equation} 
		For any $x \in \overline{B_{\frac{r}{2}}}$, if $|\phi_t(x)|>r$ for some $t\in \left]0,t_+(x)\right[$, then there exist $0 < t_1 < t_2 < t$ fulfilling 
		\begin{equation}\label{eq:conti}
			|\phi(t_1,x)|=\tfrac r 2,\quad  |\phi(t_2,x)| = r, \quad \phi([t_1,t_2],x)\subseteq \{\tfrac r 2 \leq |x|\leq r\},
		\end{equation}
		\begin{equation}
			\label{eq:cauchyphi}
			\frac r 2 \leq |\phi_{t_1}(x)- \phi_{t_2}(x)| \leq \frac{t_2-t_1}{\delta}.
		\end{equation}
	\end{corollary}
	
	\begin{proof}
Note that	\eqref{eq:conti} follows by continuity of the flow and \eqref{eq:cauchyphi} follows from \eqref{eq:cauchygen}-\eqref{eq:conti}.
	\end{proof}
	
The next lemma provides an important continuity statement:\	if the trajectory starting from a point $x_0$ gets arbitrarily close to the origin, then trajectories starting at points close to $x_0$ will also get arbitrarily close to the origin.
	
	\begin{lemma}\label{lemma:contragen}
Assume $\{x_j\}_j \subset \mb B^n\setminus \{0\}$ is a sequence converging to a point $x_0 \in \mb B^n$ such that either $x_0=0$ or there exists $T > 0$ such that $\inf_{t \in [0,T[ }|\phi_t(x_0)|=0$. Then, if $x_0\neq 0$, we have $T = t_+(x_0)$. Setting also $T = 0$ in the case $x_0 = 0$, if $0\le T_j < t_+(x_j)$ and $T_j\to T$, we have
		\begin{equation}\label{eq:conv}
			\lim_{j \to \infty}\phi_{T_j}(x_j) = 0.
		\end{equation}
	\end{lemma}
\begin{proof}
Suppose by contradiction that \eqref{eq:conv} fails. We would then find $\e > 0$ such that
	\begin{equation}\label{eq:contrale}
		|\phi_{T_j}(x_j)|\ge \eps > 0, \quad \forall j.
	\end{equation}
	Above and in what follows we do not relabel subsequences.\ If $x_0 \neq 0$ then clearly $T = t_+(x_0)$ and there exists a sequence of times $r_j \to T$ such that
	\[
	|\phi_{r_j}(x_0)| \le \frac{\eps}{4}, \quad \forall j.
	\]
Using the lower semicontinuity of $t_+$, we see that $\lim_{j\to \infty}\phi_{t}(x_j) = \phi_t(x_0)$ for each $t < T = t_+(x_0) \le \liminf_jt_+(x_j)$. Thus, through a diagonal argument we can find an increasing sequence $s_j \to T$ with $s_j < T_j$ for all $j$ such that
	\begin{equation}\label{eq:contrafin5}
		|\phi_{s_j}(x_j)| < \frac{\eps}{2}, \quad \forall j.
	\end{equation}
	If $x_0 = 0$, simply set $s_j = 0$, and notice that \eqref{eq:contrafin5} holds due to $\phi_0(x_j) = x_j \to x_0 = 0$.\ In either case, we apply Corollary \ref{cor:tec2} to find sequences $[t_j,t_j']\subset [s_j,T_j]$ such that $t_j,t_j'\to T$ and
	\[
	|\phi(t_j,x_j)|= \frac\eps 2, \quad |\phi(t_j',x_j)| = \eps , \quad \phi([t_j,t_j'],x_j) \subseteq \{\tfrac \eps 2 \leq |x|\leq \eps\}.
	\]
	By \eqref{eq:cauchyphi} we then get a $\delta > 0$ such that, for all $j \in \N$,
	\begin{equation}\label{eq:contrfin6}
		\frac{\e}{2} \le |\phi_{t_j'}(x_j)-\phi_{t_j}(x_j)| \le \frac{|t_j - t_j'|}{\delta}.
	\end{equation}
	This is a contradiction for large $j$, as $t_j,t_j' \to T$ and $T < +\infty$ in both cases, see Lemma \ref{lemma:tec1}.
\end{proof}

\subsection{Main lemmas}\label{sub:male}

To use $\phi_t$ to deform the sublevel sets of $u$, we first need a suitable neighborhood of the origin:

		\begin{lemma}[Construction of $W$]\label{lemma:ngbd}
	Let $u\in C^1(\mb B^n)$ satisfy \eqref{eq:assum}.
	There is $\e>0$ and a neighborhood $W\subset  \mb B^n$ of $0$ such that:
	\begin{enumerate}
	\item\label{it:posinv}  $W$ is a neighborhood of 0, closed in $\mb B^n$, such that $W\setminus\{0\}$ is positively invariant\footnote{The need to exclude 0 in this definition stems from the technical issue that the vector field whose flow we are considering is not defined at that point.} for $\{\phi_t\}$, i.e.\ if $x\in W\setminus\{0\}$ and $t\in [0,t_+(x)[$ then also $\phi_t(x) \in W\setminus\{0\}$;
	\item\label{it:complete} $\{-\eps \le u \leq \e\}\cap W\subseteq\tfrac 1 2 \overline{\mb B^n}$ and so in particular this set is compact.
	\end{enumerate}
	\end{lemma}

\begin{proof}
Consider $\delta= \delta(\tfrac 1 2)>0$, defined as in \eqref{eq:delta}. Take $A\equiv \overline{ \frac{1}{4}\mathbb{B}^n}\cap \{u\leq \frac{\delta }{32}\}$, and let $W$ be the closure in $\mb B^n$ of the set
$$Y\equiv \left\{\phi_t(x): x \in A\setminus\{0\}, t\in [0,t_+(x)[\right\}.$$
Since $W\setminus \{0\}$ is the closure of $Y$ in $\mb B^n\setminus \{0\}$, and since the (relative) closure of a positively invariant set is also positively invariant, clearly \ref{it:posinv} holds. We now prove that \ref{it:complete} holds with $\e=\tfrac{\delta }{128}$. In fact, it suffices to show that 
\[
\{-2\e<u<2\e\}\cap Y \subseteq \frac{1}{2} \overline{\mathbb{B}^n}.
\]
By contradiction, let $y \in \{-2\e<u<2\e\}\cap Y$ with $|y|> 1/2$. By definition of $Y$ there would then be $x\in A\setminus\{0\}$ with $|x| \le 1/4$ and $t \in \left]0,t_+(x)\right[$ such that $y =\phi_t(x) \notin {\frac{1}{2}\overline{\mathbb{B}^n}}$. Thus Corollary \ref{cor:tec2} yields $[t_1,t_2]\subset \left[0,t\right[$ fulfilling \eqref{eq:conti}-\eqref{eq:cauchyphi}. Using that $x \in A$ so that $u(x) \le \tfrac {\delta}{32}$, we deduce
$$-2\e < u(y) = u(\phi_t(x)) \le u(\phi_{t_2}(x)) \overset{\eqref{eq:cauchyuphi},\eqref{eq:cauchyphi}}{\leq} u(\phi_{t_1}(x)) - \frac{\delta }{16} \leq u(x) - \frac{\delta}{16} \leq - \frac{\delta }{32} = - 4 \e$$
which gives the desired contradiction.
\end{proof}	

The next lemma, which is the key technical tool of this section, gives a precise statement to the effect that the topology of sublevel sets can only change while crossing a critical value. For simplicity let us introduce the notation $W^b\equiv \{u\leq b\} \cap W$ and $K_b \equiv \{0\}\cap \{u=b\}$. Note that, by \eqref{eq:assum}, $K_b$ is non-empty only if $b=0$.
	
\begin{lemma}[Deformation lemma]\label{lemma:def}
Let $u\in C^1(\mb B^n)$ be as in \eqref{eq:assum}.  We assume that:
\begin{enumerate}
\item\label{it:propW} $W\subset \mb B^n$ is a closed neighborhood of 0 for which $W\setminus \{0\}$ is positively invariant for $\{\phi_t\}$;
\item\label{it:nocp} $\{a\leq u \leq b\}\cap W$ is compact and $0\not \in \{a<u<b\}$, where $a<b$.
\end{enumerate}
 Then $W^a$ is a deformation retract of $W^b \setminus K_b$.
\end{lemma}		
	
\begin{proof}
We split the proof into several steps.
\smallskip \\ 
\textbf{Step 1.} \textit{Let $x\in (\{a< u \leq b\} \cap W) \setminus K_b$. Then there is a unique arrival time $t_a(x)\in [0,t_+(x)]$ such that $\lim_{t\nearrow t_a(x)} u(\phi_t(x)) = a$.}
\newline \indent From \eqref{eq:cauchyuphi} we see that  $t\mapsto u(\phi_t(x))$ is strictly decreasing, and so $\ell\equiv \lim_{t\nearrow t_+(x)} u(\phi_t(x))$ exists. We only need to show that we cannot have $\ell>a$. 
Hence assume by contradiction that $\ell>a$, so that we could find $a<a'<\ell<b'<b$. Consider $K \equiv W \cap \{a' \leq u \leq b'\}$, which is a compact set not containing $0$ by \ref{it:nocp}. Then, if $t_1$ is such that $u(\phi_{t_1}(x)) \leq b'$, we have  $\phi_t(x) \in K$ for all $t_1 \le t < t_+(x)$. We then deduce from \eqref{eq:cauchygen} that
\[
|\phi_{s}(x)-\phi_{t}(x)|\leq \frac{s-t}{\min_K|\D u|}, \quad \text{whenever } t_1 \le t \le s <  t_+(x)
\]
Thus, since $t_+(x) < + \infty$ by Lemma \ref{lemma:tec1}, we obtain that $x_0 \equiv \lim_{t\nearrow t_+(x)} \phi_t(x)$ exists and belongs to the compact set $K$. Since $0 \notin K$, the flow can be continued at $x_0$, contradicting the definition of the maximal time of existence $t_+(x)$.
\smallskip \\
\textbf{Step 2.} \textit{$\phi_{t_a(x)}(x) \equiv \lim_{t\nearrow t_a(x)} \phi_t(x)$ exists for $x\in (\{a< u \leq b\} \cap W)\setminus K_b$.}

There are two cases to consider: writing $L \equiv \inf_{t\in [0, t_a(x)[} |\phi_t(x)|$, either $L>0$ or $L=0$. In the former case we set $K \equiv (\{a\leq u\leq b\}\cap W)\setminus B_L(0)$, which is a compact set not containing $0$ by \eqref{eq:assum}-\ref{it:nocp}, and use \eqref{eq:cauchygen} to find that for all  $0\leq t_1\leq t_2<t_a(x)$:
$$|\phi_{t_2}(x)- \phi_{t_1}(x)| \leq \frac{t_2-t_1}{\min_K|Du|},$$
so that $\lim_{t\nearrow t_a(x)} \phi_t(x)$ exists. On the other hand, if $L=0$ then $\lim_{t\nearrow t_a(x)}\phi_t(x) = 0$ by Lemma \ref{lemma:contragen} applied with a constant sequence of points.
\smallskip\\
\textbf{Step 3.} \textit{Setting $t_a(x)=0$ if $x\in \{u\leq a\}\cap W$, the function $x\mapsto t_a(x)$ is continuous in $W^b\setminus K_b$.}

Let us first note that, if $x_j\in (\{a<u\leq b\}\cap W)\setminus K_b$ is a sequence such that $x_j\to x_0$ and $u(x_j)\to a$, then by \eqref{eq:cauchyuphi} we have $$\tfrac{t_a(x_j)}{4} \leq u(x_j) - u(\phi_{t_a(x_j)}(x_j)) = u(x_j) - a \to 0.$$ Thus it suffices to prove continuity of $x\mapsto t_a(x)$ at points $x_0\in (\{a<u\leq b\}\cap W)\setminus K_b$.

Suppose first that $\phi_{t_a(x_0)}(x_0)\neq 0$. In this case we have, by \eqref{eq:pgvf} and \eqref{eq:flow},
$$\frac{\d}{\d t} u(\phi_t(x_0))\Big|_{t=t_a(x_0)} =- \frac{\langle \D u,X\rangle}{|X|^2}(\phi_{t_a(x_0)}(x_0))\leq - \frac 1 4,$$
thus the implicit function theorem shows that $t_a$, as the unique solution of $u(\phi_{t_a(x)}(x))=a$, is $C^1$ in a neighborhood of $x_0$. 

Suppose now that $\phi_{t_a(x_0)}(x_0) =0$. If $x\mapsto t_a(x)$ were not continuous at $x_0$ there would be a sequence $x_j\to x_0$ and $\e>0$ such that $|t_a(x_j)-t_a(x_0)|\geq \e>0$. Consider first the case $t_a(x_j) \le t_a(x_0) - \e$ for all $j$ (up to passing to an unrelabeled subsequence). We have
\begin{equation}\label{eq:strict}
t_a(x_j) \le t_a(x_0) - \e < t_+(x_0) \le \liminf_jt_+(x_j),
\end{equation}
by lower semicontinuity of $t_+$.\ By monotonicity of $t \mapsto u(\phi_t(x_j))$, we also have
\begin{equation}\label{eq:a}
u(\phi_{t_a(x_0)-\e}(x_j)) \le u(\phi_{t_a(x_j)}(x_j)) = a.
\end{equation}
Because of \eqref{eq:strict}, we can use continuity of $x \mapsto \phi_{t_a(x_0)-\e}(x)$ in a neighborhood of $x_0$, to deduce
\[
u(\phi_{t_a(x_0)-\e}(x_0)) = \lim_{j} u(\phi_{t_a(x_0)-\e}(x_j)) \overset{\eqref{eq:a}}{\le} a,
\]
which is in contradiction with the definition of $t_a(x_0)$. Consider now the case where $t_a(x_j) \ge t_a(x_0) + \e$ for all $j$, again up to unrelabeled subsequences.
We claim that 
\begin{equation}\label{eq:cl2}
	\lim_ju(\phi_{t_a(x_0)}(x_j)) = a.
\end{equation}
This would yield a quick contradiction, since then
\[                                                                                        
 a-u(\phi_{t_a(x_0)}(x_j)) = u(\phi_{t_a(x_j)}(x_j))- u(\phi_{t_a(x_0)}(x_j)) \overset{\eqref{eq:cauchyuphi}}{\le} -\frac{t_a(x_j)-t_a(x_0)}{4} \le -\frac{\eps}{4}
\] 
is false for $j$ large by \eqref{eq:cl2}. Let us show \eqref{eq:cl2}, again reasoning by contradiction. Since $t_a(x_j) \ge t_a(x_0) + \e$, by monotonicity we deduce that $u(\phi_{t_a(x_0)}(x_j)) \ge a$ for all $j$, and hence contradicting \eqref{eq:cl2} leads us to a $\delta > 0$ such that, for all $j$,
\begin{equation}\label{eq:contrcontr}
u(\phi_{t_a(x_0)}(x_j)) \ge a + \delta.
\end{equation}
But then, for all $k > 0$ and for all $j$,
\[
0 \le u(\phi_{t_a(x_0)- k}(x_j)) - u(\phi_{t_a(x_0)}(x_j)) \overset{\eqref{eq:contrcontr}}{\le} u(\phi_{t_a(x_0)- k}(x_j)) - a -\delta.
\]
By lower semicontinuity of $t_+$, we find that $x\mapsto \phi_{t_a(x_0)-k}(x)$ is continuous in a neighborhood of $x_0$, for each fixed $k > 0$. Thus, letting first $j \to \infty$ and then $k \to 0$, we reach the contradiction $0 \le - \delta$, and hence \eqref{eq:cl2} is proved. This concludes the proof of this step.

\smallskip
\textbf{Step 4.}\textit{ The map $h\colon [0,1]\times (W^b \setminus K_b)\to W^b \setminus K_b $ defined by
$$h(t,x)\equiv \begin{cases}
x & \text{if } x\in \{u\leq a\} \cap W = W^a,\\
\phi(t_a(x) t, x) & \text{if } x\in (\{a<u\leq b\} \cap W)\setminus K_b,
\end{cases}$$
for $t\in [0,1]$,  gives a deformation retract of $W^b \setminus K_b$ to $W^a$.}
Note that $h(t,W^b\setminus K_b)\subseteq W^b\setminus K_b$ due to the positive invariance \ref{it:propW} of $W\setminus\{0\}$.\ Obviously we have $h(0,\cdot) = \id_{W^b\setminus K_b}$, $h(1,W^b\setminus K_b) \subseteq W^a$ and $h(t,\cdot)|_{W^a} = \id_{W^a}$ for all $t$, so the only thing to check is that $h$ is continuous.


We want to check that $h$ is continuous at the point $(t_0,x_0) \in [0,1]\times (W^b\setminus K_b)$. There are the following cases to consider: 
\begin{enumerate}
	\item $u(x_0) < a$;\label{c:1}
	\item $u(x_0) = a$ and $x_0 \neq 0$;\label{c:2}
	\item $u(x_0) = a$, $x_0 = 0$;\label{c:3}
	\item $t_0 \in [0,1)$, $a < u(x_0)$; \label{c:4}
	\item $t_0 = 1$, $a < u(x_0)$, $\lim_{t\nearrow t_a(x_0)}\phi_t(x_0) \neq 0$;\label{c:5}
	\item $t_0 = 1$, $a < u(x_0)$, $\lim_{t\nearrow t_a(x_0)}\phi_t(x_0) = 0$.\label{c:6}
\end{enumerate}
Observe that the existence of the above limits is provided by Step 2.
Case \ref{c:1} is immediate, since in a neighborhood of $(t_0,x_0)$ we have $h(t,x) = x$. Cases \ref{c:2}-\ref{c:4}-\ref{c:5} are all treated analogously, and the proof rests on an application of Step 3 and on noticing that, under any of those assumptions, there is a neighborhood of $(t_0,x_0)$ on which $(s,x) \mapsto \phi_{st_a(x)}(x)$ is continuous. We are only left with Cases \ref{c:3}-\ref{c:6}, which are shown similarly. Consider first \ref{c:3}. If $h$ were not continuous at $(t_0,x_0)= (t_0,0)$, then there would exist $\eps > 0$ and sequences $(t_j,x_j) \to (t_0,0)$ such that
\begin{equation}\label{eq:contrafin1}
|h(t_j,x_j) - h(t_0,0)| = |\phi_{t_jt_a(x_j)}(x_j)| \ge \e > 0,\quad \forall j.
\end{equation}
Here we have used the fact that $x_j$ cannot fulfill $u(x_j) < a$ for infinitely many $j$, since otherwise \eqref{eq:contrafin1} is trivially false. Observe that $t_a(x_j) \to 0 = t_a(0)$ by Step 3. With this observation, we can immediately find a contradiction to \eqref{eq:contrafin1} through Lemma \ref{lemma:contragen}.
Case \ref{c:6} is almost identical:  were the conclusion false, we would have again $\eps > 0$ and sequences $(t_j,x_j) \to (1,x_0)$ such that
\begin{equation}\label{eq:contrafin2}
	|h(t_j,x_j) - h(1,x_0)| = |\phi_{t_jt_a(x_j)}(x_j) - \phi_{t_a(x_0)}(x_0)| = |\phi_{t_jt_a(x_j)}(x_j)| \ge \e > 0.
\end{equation}
Observing that, again by Step 3, $t_jt_a(x_j) \to t_a(x_0)$ and by assumption $\phi_{t_a(x_0)}(x_0) = 0$, we are in position to apply Lemma \ref{lemma:contragen} to readily discard \eqref{eq:contrafin2}. Thus, $h$ is continuous, and we conclude the proof of this step and of this lemma.
\end{proof}

\subsection{Properties of the critical groups}\label{sub:crit}

As a consequence of the machinery we developed so far, we can easily prove the following:

\begin{proposition}\label{prop:nicecritgroup}
Let $u\in C^1(\mb B^n)$ be as in \eqref{eq:assum} and $\e,W$ satisfy \ref{it:posinv}-\ref{it:complete}of Lemma \ref{lemma:ngbd}. Then,
$$C_{k}(u, 0) = H_k(W^\e, W^{-\e})$$
\end{proposition}	
	
\begin{proof}
By the excision property of the relative homology we can compute the critical groups with $U=W$, i.e.\ $C_k(u,0)\cong H_k(W^0, W^0\setminus \{0\})$.
By Lemmas \ref{lemma:def} and  \ref{lemma:ret}\ref{it:def1}  we have $H_k(W^\e,W^{-\e}) \cong H_k(W^0,W^{-\e})$, while  Lemmas \ref{lemma:def} and  \ref{lemma:ret}\ref{it:def2}  give $H_k(W^0,W^{-\e})\cong H_k(W^0, W^0\setminus \{0\})$.
\end{proof}
	
\begin{corollary}[Constancy of the critical groups in $C^1$]\label{cor:continuity}
Let $u_1, u_2\in C^1(\mb B^n)$ be two functions satisfying \eqref{eq:assum}. There is $\gamma>0$, depending only on $u_1$, such that 
$$\|u_1-u_2\|_{C^1(\mb B^n)} \leq \gamma\quad \implies \quad C_k(u_1, 0)= C_k(u_2,0).$$
\end{corollary}

\begin{proof}
Let $W,\e$ be as in Lemma \ref{lemma:ngbd}, applied with $u=u_1$, and take $r>0$ so that 
\begin{equation}
\label{eq:chooser}
B_r(0)\subseteq \{-\tfrac \e 2\leq u \leq \tfrac \e 2\}\cap W.
\end{equation}
Take a cutoff $\eta\in C^\infty_c(B_r(0))$  with $0\leq \eta \leq 1$, $\eta=1$ in $B_{r/2}(0)$ and $|\D \eta|\leq Cr^{-1}$, for a dimensional constant $C$, and set 
$$\widetilde u_2 \equiv u_1 + \eta (u_2-u_1),$$
thus $\widetilde u_2= u_1$ outside $B_r(0)$. Since $|\widetilde u_2 - u_1| \leq \gamma$, if we choose $\gamma < \frac \e 2$ then by \eqref{eq:chooser} we have 
\begin{equation}
\label{eq:samesublevelsets}
\{u_1\leq \pm \e\}=\{\widetilde u_2\leq\pm \e\}
\end{equation}
and so $\{-\e \leq \widetilde u_2 \leq \e\}\cap W = \{-\e \leq u_1 \leq \e\}\cap W$ is a compact set. 
As usual let us take $0<\delta \equiv \inf_{B_r(0)\setminus B_{r/2}(0)} |\D u|$.  We then estimate, in $B_r(0)\setminus B_{r/2}(0)$, 
$$
|\D \widetilde u_2| \geq |\D u_1| - \eta |\D u_1 - \D u_2| - |\D \eta| |u_1 -u_2|\geq \delta -(1+Cr^{-1})\gamma \geq  \frac \delta 2,
$$
provided that we also choose $\gamma \leq \frac{\delta}{2(1+Cr^{-1})}$. Since $\widetilde u_2=u_2$ in $B_{r/2}(0)$ and $\widetilde u_2=u_1$ outside $B_r(0)$, we see that $\widetilde u_2$ still satisfies \eqref{eq:assum}. Moreover $W \setminus \{0\}$ is also positively invariant for a suitable pseudo-gradient flow of $\widetilde u_2$ which matches the one of $u_1$ outside $B_r(0)$, since $B_r(0)$ is contained in the interior of $W$. Thus Proposition \ref{prop:nicecritgroup} is applicable also to $\widetilde u_2$, and we infer:
\[
C_k(\widetilde u_2,0) = H_k(W\cap \{\widetilde u_2 \le \eps\},W\cap \{\widetilde u_2 \le -\eps\}) \overset{\eqref{eq:samesublevelsets}}{=} H_k(W\cap \{u_1 \le \eps\},W\cap \{u_1 \le -\eps\}) = C_k(u_1,0).
\]
In addition, since $\widetilde u_2 = u_2$ in $B_{\frac{r}{2}}(0)$, we can exploit the independence of the critical groups on the choice of neighborhood, see Definition \ref{def:crit}, to write
\begin{align*}
C_k(u_2,0) &= H_k(B_\frac{r}{2}(0)\cap \{u_2 \le 0\},B_\frac{r}{2}(0)\cap \{u_2 \le 0\}\setminus \{0\}) \\
&= H_k(B_\frac{r}{2}(0)\cap \{\widetilde u_2 \le 0\},B_\frac{r}{2}(0)\cap \{\widetilde u_2 \le 0\}\setminus \{0\}) = C_k(\widetilde u_2,0),
\end{align*}
and conclude the proof.
\end{proof}

We conclude this section computing the critical groups of some simple but important examples:

\begin{proposition}[Examples of critical groups]\label{prop:ex}
Let $u\in C^1(\mb B^n)$.
\begin{enumerate}
\item\label{it:min} If $x_0$ is an isolated local minimum of $u$ then $C_k(u,x_0) = \delta_{k,0} \Z$.
\item\label{it:nomin} If $x_0$ is an isolated critical point which is not a local minimum of $u$ then $C_0(u,x_0) = \{0\}$.
\item\label{it:quad} For non-singular $A\in \Sym(n)$  and $q_A(x) \equiv \tfrac 1 2\langle Ax,x\rangle$ we have $C_k(q_A,0) = \delta_{k,\ind(A)} \Z$.
\item\label{it:morse} If $x_0$ is an isolated critical point of $u$ and $u$ is twice differentiable at $x_0$ with $A\equiv \D^2 u(x_0)$ non-singular, then $C_k(u,x_0) = \delta_{k,\ind(A)} \Z.$
\end{enumerate}
\end{proposition}

\begin{proof}
Since the calculation of the critical groups is local we can suppose without loss of generality that $x_0=0$ and, by adding a constant to $u$,  that  \eqref{eq:assum} holds. 

Observe that \ref{it:min} is a direct consequence of the dimension axiom \ref{it:dim}. 

For \ref{it:nomin} we can  apply Lemma \ref{lemma:ngbd} to find $\e>0$ and $W$  such that 
$$C_0(u,0)=H_0(W^0, W^{0}\setminus \{0\}) = H_0(W^\e,W^0\setminus \{0\}),$$
where the last equality follows from Lemmas \ref{lemma:ret}\ref{it:def1} and \ref{lemma:def}, since there is a deformation retract $h\colon [0,1]\times W^\e \to W^\e$ which deforms $W^\e$ to $W^0$.  To see that $H_0(W^\e,W^0\setminus \{0\})=\{0\}$ it suffices to note that every point $x\in W^\e$ can be connected to a point in $W^0\setminus \{0\}$ through a path in $W^\e$: indeed, $h(\cdot,x)\colon [0,1]\to W^\eps$ is a path between $x$ and a point in $W^0$; moreover $0\in \tp{int}(W)$, not being a local minimum, can itself be connected to points in $W^0\setminus \{0\}$ through paths in $W^\e$.

For \ref{it:quad}, by changing coordinates we can suppose that $A=\tp{diag}(-1,	\dots,-1,1\dots, 1)$ so let us write $\ell= \ind(A)$, $x_-=(x_1,\dots, x_\ell,0\dots, 0)$ and $x_+=(0,\dots, 0, x_{\ell+1},\dots, x_n)$, thus $q_A(x) =\tfrac 1 2 ( |x_+|^2 - |x_-|^2)$. By considering the function $h\colon [0,1]\times \mb B^n\to \mb B^n, (t,x)\mapsto  x_- +(1-t) x_+$ we see that $\{x_+=0\}\cap \mb B^n$ is a deformation retract of $\{q_A\leq 0\}\cap \mb B^n$, and that $\{x_+=0\}\cap \mb B^n\setminus \{0\}$ is a deformation retract of $\{q_A\leq 0\}\cap \mb B^n\setminus \{0\}$. Thus, applying Lemmas \ref{lemma:ret} and \ref{lemma:computationhom}, we have 
\begin{align*}
H_k(\{q_A\leq 0\}\cap \mb B^n, \{q_A\leq 0\}\cap \mb B^n\setminus \{0\})& 
 \cong H_k(\{x_+=0\} \cap \mb B^n, \{x_+=0\}\cap \mb B^n\setminus \{0\}) \\
 & \cong H_k(\mb B^\ell, \mb B^\ell\setminus \{0\})\cong \delta_{k,\ell} \Z;
\end{align*}
note that by the dimension axiom \ref{it:dim} this calculation holds even when $\ell=0$.

For \ref{it:morse}, the assumption that $0$ is a point of twice differentiability of $u$ guarantees that the sequence of rescalings $u_r(x) \equiv u(r x)/r^2$ converges to $q_A$ in $C^1(\mb B^n)$, as $r\to 0$.  Since $u_r$ still satisfy \eqref{eq:assum}, by \ref{it:quad} and Corollary \ref{cor:continuity} we have $C_k(u_r,x_0) = \delta_{k,\ind(A)} \Z$ for all $r$ small enough. But clearly $C_k(u_r, 0)$ is independent of $r$, so the claim follows.
\end{proof}

\section{Proofs of the main results}

Let $\Omega\subset \R^n$ be a connected open set.  We begin with the following:

\begin{theorem}[Constancy of the critical groups]\label{thm:constcritgroup}
Let $u\in C^1(\Omega)$ be such that $\D u$ is locally injective. For $x_0\in \Omega$, let 
$$u_{x_0}(x) \equiv u(x)- u(x_0) - \langle \D u(x_0), x-x_0\rangle.$$ Then
$ C_k(u_{x_0}, x_0) $ is independent of $x_0\in \Omega$. 
\end{theorem}

\begin{proof}
Fix an arbitrary $x_0\in \Omega$; it suffices to show that $y\mapsto C_k(u_y,y)$ is constant in a neighborhood of $x_0$. By changing coordinates we can suppose that $x_0=0$ and that $\D u$ is injective in $4\mb B^n$. Let us write 
$$\widetilde u_y(x) \equiv u_y(y+ x) = u(y+x) - u(y) - \langle \D u(y),x\rangle,$$
so that $\widetilde u_y$ satisfies \eqref{eq:assum} for all $y\in \mb B^n$.
If $\omega$ is a modulus of continuity for $\D u$ in $4\mb B^n$ and $L\equiv \sup_{4\mb B^n} |\D u|$, then it is easy to see that
 $$\sup_{x\in \mb B^n} \left[|\widetilde u_y-\widetilde u_0|(x)+|\D \widetilde u_y- \D \widetilde u_0|(x) \right]\leq 2L|y|  + 3\omega(|y|),$$
 and the right-hand side vanishes as $|y|\to 0$.
Thus there is a small $\delta>0$, depending only on $\widetilde u_0$, such that for $|y|\leq \delta$ we can apply Corollary \ref{cor:continuity} to deduce that $C_k(\widetilde u_y,0)=C_k(\widetilde u_0,0)$.
Since the set $ \{u_y\leq 0\}$ is just a translation of  $\{\widetilde u_y\leq 0\}$,  we infer that the groups $C_k(u_y,y)$ are independent of $y$ whenever $|y|\leq \delta$. 
\end{proof}

As a consequence of Theorem \ref{thm:constcritgroup} and the results of Section \ref{sec:morse}, we obtain a new proof of a theorem due to Ball \cite{Ball1980a}.
It is easy to see that if $\Omega$ is convex and if $u\in C^1(\Omega)$ is a strictly convex function then its gradient is injective. In fact,  the converse is also true, in a strong form: if $\D u$ is locally injective and if $u$ admits a locally supporting hyperplane at a single point (condition \eqref{eq:suphyp} below) then $u$ is strictly convex:

\begin{corollary}[Ball \cite{Ball1980a}]\label{cor:ball}
Assume that $\Omega$ is convex and let $u\in C^1(\Omega)$ be such that $\D u$ is locally injective. If there is $x_0\in \Omega$ such that
\begin{equation}
\label{eq:suphyp}
u_{x_0}\geq 0 \text{ in a neighborhood of } x_0, 
\end{equation}
then $u$ is strictly convex. 
\end{corollary}

\begin{proof}
Since $u_{x_0}(x_0)=0$, \eqref{eq:suphyp} asserts that $x_0$ is a local minimum of $u_{x_0}$, and in fact it must be isolated: if not, there would be $x_j\to x_0$ with $x_j$ local minima of $u_{x_0}$, thus we would have $0 = \D u_{x_0}(x_j) = \D u(x_j) - \D u(x_0)$, contradicting the local injectivity of $\D u$ at $x_0$. Hence, by Proposition \ref{prop:ex}\ref{it:min} we see that $C_{k}(u_{x_0},x_0) = \delta_{k,0}\Z$ and so by Theorem \ref{thm:constcritgroup} we must in fact have $C_k(u_y,y)= \delta_{k,0} \Z$ for all $y\in \Omega$. Since $y$ is always a critical point of $u_y$, by Proposition \ref{prop:ex}\ref{it:nomin} we deduce that every $y\in \Omega$ is a local minimum of $u_y$, which as before must be isolated. This implies that $u$ is strictly convex in $\Omega$: if not, there would be distinct $y_0,y_1\in \Omega$ and $\lambda \in (0,1)$ such that
$$u(y_\lambda ) \geq \lambda u(y_1) + (1-\lambda) u(y_0), \qquad y_\lambda \equiv \lambda y_1 + (1-\lambda) y_0.$$
But then the function $\varphi \colon t\mapsto u(y_t ) - [t u(y_1) + (1-t) u(y_0)]$ attains a maximum at some $t^*\in\left]0,1\right[$ and thus  $\varphi'(t^*)=0$, or equivalently $\langle \D u(y_{t^*}), y_1-y_0\rangle = u(y_1)-u(y_0)$. Hence, for $t\in [0,1]$, 
$$0 \geq\varphi(t) -  \varphi(t^*) = u(y_t) - u(y_{t^*}) - \langle \D u(y_{t^*}), y_t-y_{t^*}\rangle = u_{y_{t^*}}(y_t),$$
while for $t$ close to $t^*$ we also have, by \eqref{eq:suphyp},
\[
u_{y_{t^*}}(y_t)\geq 0.
\]
Combining these two inequalities we infer that $y_{t^*}$ is not an isolated local minimum of $u_{y_{t^*}}$, which leads to a contradiction.
\end{proof}


%





We now prove the main results of this paper. In fact we will prove the following more general version of Theorem \ref{thm:main}:

\begin{theorem}[Constancy of the index]\label{thm:mainFD}
Let $f\in W^{1,n}_\loc(\Omega,\R^n)$ be a map of finite distortion with $K_f\in L^p_\loc(\Omega)$, for $p$ an exponent satisfying \eqref{eq:condp}. If $\D f\in \Sym(n)$ a.e.\ in $\Omega$ then $\ind(\D f)$ is constant a.e.\ in $\Omega$.
\end{theorem}
	
\begin{proof}
As the result is local in nature, we can assume $\Omega$ to be a ball. Due to this, since $\D f\in \Sym(n)$ a.e.\ in $\Omega$, there is a potential $u\in W^{2,n}(\Omega)$ such that $\D u =f$. We also have that $u \in C^1(\Omega)$, since $f$ has finite distortion, compare Theorem \ref{thm:FD}. According to the same theorem, $\Omega'\equiv \Omega\setminus B_f$ is a connected open set of full measure, and $f = \D u$ is locally injective on it.\ By Theorem \ref{thm:constcritgroup} the groups $C_{k}(u_{x},x)$ are independent of $x\in \Omega'$. Let $E\equiv \{x\in \Omega':\D^2u (x) \tp{ exists  and } \det(\D^2 u(x))\neq 0\}.$ Again by Theorem \ref{thm:FD} we see that $E$ has full measure in $\Omega'$, and by Proposition \ref{prop:ex}\ref{it:morse} we infer that the index of $\D^2 u(x)$ is constant in $E$.
\end{proof}	

%
	
Similarly, we obtain the following strengthening of Theorem \ref{thm:MA}:
	
\begin{corollary}\label{cor:MAFD}
If $u\in W^{2,np}_\loc(\Omega)$ solves \eqref{eq:MA}, with  $p$ as in \eqref{eq:condp}, then $\ind(\D^2 u)$ is a.e.\ constant.
\end{corollary}
Although \eqref{eq:condp} is optimal\footnote{Up to the endpoint $p=n-1$.} for the validity of Theorem \ref{thm:FD} \cite{Guo2021}, we do not know if it is also optimal for the validity of Corollary \ref{cor:MAFD}, except when $n=2$. In fact, in this case one can use Faraco's staircase laminates \cite{Faraco2003b,Faraco2018} to construct maps $u\in W^{2,q}_\loc(\mb B^2)$, for some $q<2$, which solve \eqref{eq:MA} but for which the index is not constant. In general dimension, if one makes the additional assumption that $u$ is affine in a neighborhood of $\partial \Omega$ then, by the results of \v Sver\'ak in \cite{Sverak1992},  Corollary \ref{cor:MAFD} holds for $u\in W^{2,n}_\loc(\Omega)$. This suggests that perhaps the following has an affirmative answer:

\begin{question}
Let $u\in W^{2,n}_\loc(\Omega,\R^n)$ solve \eqref{eq:MA}. Is $f=\D u$ an open and discrete map?
\end{question}
If the answer is positive then Corollary \ref{cor:MAFD} would hold for $u\in W^{2,n}_\loc(\Omega)$, as in the case of affine boundary conditions. 
	\let\oldthebibliography\thebibliography
	\let\endoldthebibliography\endthebibliography
	\renewenvironment{thebibliography}[1]{
	\begin{oldthebibliography}{#1}
	\setlength{\itemsep}{0.5pt}
	\setlength{\parskip}{0.5pt}
	}
	{
	\end{oldthebibliography}
	}
	
	{\small
	\bibliographystyle{abbrv-andre}
	\bibliography{library.bib}
	}

\end{document}